\documentclass[oneside]{amsart}
\usepackage{amsmath,amsfonts,amssymb}%
\usepackage[libertine]{newtxmath}

\usepackage[all]{xy}
\usepackage[usenames]{color}
\usepackage{setspace}
\usepackage{comment}
\usepackage{url}

\usepackage{enumitem}

\usepackage{amsmath, amsthm, amsfonts, amssymb, stmaryrd}%
\usepackage{mathrsfs}

\usepackage{enumerate}

\usepackage{xparse, expl3}

\usepackage[capitalize]{cleveref}  
\crefname{lemma}{Lemma}{Lemmas}
\crefname{corollary}{Corollary}{Corollaries}
\crefname{theorem}{Theorem}{Theorems}
\crefname{equation}{Equation}{Equations}
\crefname{example}{Example}{Examples}
\crefname{section}{Section}{Sections}
\crefname{subsection}{Section}{Sections}

\def\xx{{\EM{\mbf{x}}}}
\def\yy{{\EM{\mbf{y}}}}
\def\zz{{\EM{\mbf{z}}}}

\def\Aut{{\mathrm{{Aut}}}}

\def\Lempty{\EM{\mc{L}}}

\def\Lomega#1{\EM{\Lempty_{#1, \w}}}

\def\Lkappalambda#1#2{{\EM{\Lempty_{#1, #2}}}}

\def\Liw{\Lomega{\infty}}

\def\Lik{{\EM{\Lkappalambda{\infty}{\kappa}}}}

\newcommand\rest[1][]{\EM{\!\restriction_{#1}}}

\newcommand{\defn}[1]{{\bf{#1}}}

\def\ZFC{\text{ZFC}}
\def\Set{\text{SET}}

\newcommand{\Lowenheim}{\text{L\"{o}wenheim}}
\newcommand{\LS}{\text{\Lowenheim-Skolem}}

\DeclareMathOperator{\ORD}{ORD}

\DeclareDocumentCommand{\forces}{d[]}
{
\IfNoValueTF{#1}
	{
	\EM{\Vdash}
	}
	{
	\EM{\Vdash_{#1}}
	}
}

\DeclareDocumentCommand{\dual}{d()}
{
\IfNoValueTF{#1}
	{
	\EM{\hat{\ }}
	}
	{
	\EM{\hat{#1}}
	}
}

\DeclareDocumentCommand{\SizeAS}{d()}
{
\IfNoValueTF{#1}
	{
	\EM{|\cdot|}
	}
	{
	\EM{|#1|}
	}
}

\DeclareDocumentCommand{\compcK}{d[]}
{
\IfNoValueTF{#1}
	{
	\EM{\mathbb{K}}
	}
	{
	\EM{\mathbb{K}[#1]}
	}
}

\DeclareDocumentCommand{\AgeK}{d<> d[] d()}
{
\IfNoValueTF{#1}
{
    \IfNoValueTF{#2}
    {
        \IfNoValueTF{#3}
        {
            \EM{\mbf{K}}%
        }
        {
            \EM{\mbf{K}^{#3}}%
        }
    }
    {
        \IfNoValueTF{#3}
        {
            \EM{\mbf{K}[#2]}%
        }
        {
            \EM{\mbf{K}^{#3}[#2]}%
        }
    }
}
{
    \IfNoValueTF{#2}
    {
        \IfNoValueTF{#3}
        {
            \EM{\mbf{K}_{#1}}%
        }
        {
            \EM{\mbf{K}^{#3}_{#1}}%
        }
    }
    {
        \IfNoValueTF{#3}
        {
            \EM{\mbf{K}_{#1}[#2]}%
        }
        {
            \EM{\mbf{K}_{#1}^{#3}[#2]}%
        }
    }
}
}

\DeclareDocumentCommand{\AgeKBot}{}
{
\AgeK(-)
}
\DeclareDocumentCommand{\LangBot}{}
{
    \Lang^-%
}
\DeclareDocumentCommand{\MBot}{}
{
    \cM^-%
}
\DeclareDocumentCommand{\NBot}{}
{
    \cN^-%
}
\DeclareDocumentCommand{\DBot}{}
{
    \cD^-%
}

\DeclareDocumentCommand{\Rel}{d[]}
{
\IfNoValueTF{#1}
	{
	\EM{\mathcal{R}}
	}
	{
	\EM{\mathcal{R}_{#1}}
	}
}

\DeclareDocumentCommand{\Func}{d[]}
{
\IfNoValueTF{#1}
	{
	\EM{\mathcal{F}}
	}
	{
	\EM{\mathcal{F}_{#1}}
	}
}

\DeclareDocumentCommand{\Sort}{d[]}
{
\IfNoValueTF{#1}
	{
	\EM{\mathcal{S}}
	}
	{
	\EM{\mathcal{S}_{#1}}
	}
}

\DeclareDocumentCommand{\arity}{}{\textbf{ar}}
\DeclareDocumentCommand{\ar}{d[]}
{
\IfNoValueTF{#1}
	{
	\EM{\arity}
	}
	{
	\EM{\arity_{#1}}
	}
}

\DeclareDocumentCommand{\Fn}{d<> d[] d()}
{
\IfNoValueTF{#3}
	{
	\EM{\textrm{Fn}(#1, #2)}
	}
	{
	\EM{\textrm{Fn}(#1, #2, #3)}
	}
}

\DeclareDocumentCommand{\UnderSet}{m}
{
\EM{\mbf{set}(#1)}
}

\DeclareMathOperator{\Tree}{Tr}

\DeclareMathOperator{\Sym}{Sym}
\DeclareDocumentCommand{\Perm}{d()}
{
\EM{\Sym(#1)}
}

\DeclareDocumentCommand{\name}{m}
{
\EM{\dot{#1}}
}

\newcommand{\Fraisse}{\textrm{Fra\"iss\'e}}

\newcommand{\ThFraisse}{\text{Fr}}

\newcommand{\cofinal}{\EM{\textrm{cf}}}
\newcommand{\tc}{\EM{\textrm{tc}}}

\newcommand{\Field}{\mbb{F}}

\DeclareDocumentCommand{\tdcl}{d[] d()}
{
\IfNoValueTF{#1}
{
    \EM{\textbf{tcl}(#2)}
}
{
    \EM{\textbf{tcl}_{#1}(#2)}
}
}

\DeclareDocumentCommand{\gdcl}{d[] d()}
{
\IfNoValueTF{#1}
{
    \EM{\textrm{gcl}(#2)}
}
{
    \EM{\textrm{gcl}_{#1}(#2)}
}
}

\DeclareDocumentCommand{\qftp}{d[] d()}
{
\IfNoValueTF{#1}
{
    \EM{\quantfreetp(#2)}
}
{
    \EM{\quantfreetp_{#1}(#2)}
}
}

\DeclareMathOperator{\quantfreetp}{qtp}

\DeclareDocumentCommand{\Closure}{d<> d()}
{
\IfNoValueTF{#1}
{
    \EM{\textrm{cl}(#2)}
}
{
    \EM{\textrm{cl}_{#1}(#2)}
}
}

\DeclareDocumentCommand{\ClosureMap}{d[] d()}
{
\EM{\textrm{clMap}(#1, #2)}
}

\DeclareDocumentCommand{\CHP}{d()}
{
\IfNoValueTF{#1}
{
    \textrm{(CHP)}
}
{
    \textrm{(\EM{#1}-CHP)}
}
}

\DeclareDocumentCommand{\CJEP}{d()}
{
\IfNoValueTF{#1}
{
    \textrm{(CJEP)}
}
{
    \textrm{(\EM{#1}-CJEP)}
}
}

\makeatletter
\newcommand{\dotminus}{\mathbin{\text{\@dotminus}}}

\newcommand{\@dotminus}{%
  \ooalign{\hidewidth\raise1ex\hbox{.}\hidewidth\cr$\m@th-$\cr}%
}

\DeclareMathOperator{\Lang}{\mathscr{L}}
\DeclareMathOperator{\id}{id}

\def\aa{{\EM{\mbf{a}}}}
\def\bb{{\EM{\mbf{b}}}}
\def\cc{{\EM{\mbf{c}}}}
\def\dd{{\EM{\mbf{d}}}}

\def\cA{{\EM{\mathcal{A}}}}
\def\cB{{\EM{\mathcal{B}}}}
\def\cC{{\EM{\mathcal{C}}}}
\def\cD{{\EM{\mathcal{D}}}}
\def\cE{{\EM{\mathcal{E}}}}
\def\cF{{\EM{\mathcal{F}}}}
\def\cG{{\EM{\mathcal{G}}}}

\def\cM{{\EM{\mathcal{M}}}}
\def\cN{{\EM{\mathcal{N}}}}

\def\cR{{\EM{\mathcal{R}}}}
\def\cS{{\EM{\mathcal{S}}}}
\def\cT{{\EM{\mathcal{T}}}}

\def\cX{{\EM{\mathcal{X}}}}

\def\w{\EM{\omega}}

\def\^{\EM{{}^{\And}}}

\def\And{\EM{\wedge}}

\def\<{\EM{\langle}}
\def\>{\EM{\rangle}}

\def\EM#1{\ensuremath{#1}}
\def\mbb#1{\EM{\mathbb{#1}}}
\def\mbf#1{\EM{\mathop{\pmb{#1}}}}
\def\mc#1{\EM{\mathcal{#1}}}

\def\st{\,:\,}
\def\:{\colon}

\providecommand{\dotdiv}{
  \mathbin{
    \vphantom{+}
    \text{
      \mathsurround=0pt 
      \ooalign{
        \noalign{\kern-.35ex}
        \hidewidth$\smash{\cdot}$\hidewidth\cr 
        \noalign{\kern.35ex}
        $-$\cr 
      }%
    }%
  }%
}

\DeclareDocumentCommand{\RightJustify}{m}{\hspace*{\fill}\mbox{#1}\penalty-9999\relax}

\newcounter{margincounter}
\DeclareDocumentCommand{\displaycounter}{}
	{{\arabic{margincounter}}}
\DeclareDocumentCommand{\incdisplaycounter}{}
	{{\stepcounter{margincounter}\arabic{margincounter}}}

\DeclareDocumentCommand{\DeclareComment}{m m m o d()}{%
\expandafter\DeclareDocumentCommand\csname Hide#1\endcsname {}
	{%
	\expandafter\DeclareDocumentCommand\csname #1\endcsname {+m} {}
	
	\expandafter\DeclareDocumentCommand\csname f#1\endcsname {+m} {}

	\expandafter\DeclareDocumentEnvironment{e#1} {} {} {}
	}

\expandafter\DeclareDocumentCommand\csname Show#1\endcsname {}
	{
	\expandafter\DeclareDocumentCommand\csname #1\endcsname {+m}
		{%
		\textcolor{#2}
			{ 
			{\tiny \bf (#3)}
			\IfValueT{#5}
				{%
				#5
				}
			####1
			}
		}

	\expandafter\DeclareDocumentCommand\csname f#1\endcsname {+m}
		{%
		\IfValueTF{#4}
			{
			\textcolor{#2}
			{\text{$\,^{(\incdisplaycounter{#4})}$}}
			\marginpar{\tiny\textcolor{#2}{
				{\text{\tiny $(\displaycounter{#4})$}}
				\text{\IfValueT{#5}{#5}
				####1}}}
			}
			{
			\textcolor{#2}
			{$\,^{(\incdisplaycounter)}$}
			\marginpar{\tiny\textcolor{#2}{
				{\tiny $(\displaycounter)$}
				\text{\IfValueT{#5}{#5}
				####1}}}
			}
		}

	\expandafter\DeclareDocumentEnvironment{e#1} {}
		{
		\textcolor{#2}
		\bgroup
		\IfValueT{#5}
			{%
			#5
			}
		}
		{
		\egroup
		}
	}

\csname Show#1\endcsname

}

\definecolor{NAColor}{rgb}{1.0,0.0,0.0}
\definecolor{ProblemColor}{rgb}{0.7,0.1,0.7}
\definecolor{TBDColor}{rgb}{0.0,0.0,0.8}
\definecolor{MathColor}{rgb}{0.0,0.4,0.1}
\definecolor{NateColor}{rgb}{0.0,0.5,1.0}
\definecolor{MostafaColor}{rgb}{1.0,0.0,1.0}
\definecolor{RefColor}{rgb}{1.0,0.0,1.0}
\definecolor{LaterColor}{rgb}{1.0,0.0,1.0}

\DeclareComment{NA}{NAColor}{\EM{\star}}
\DeclareComment{TBD}{TBDColor}{\EM{!}}
\DeclareComment{PROBLEM}{ProblemColor}{\EM{!!}}(\bf)
\DeclareComment{MATH}{MathColor}{\EM{\#}}
\DeclareComment{NATE}{NateColor}{\EM{\square}}(N:)
\DeclareComment{MOSTAFA}{MostafaColor}{\EM{\square}}(M:)
\DeclareComment{REF}{RefColor}{\EM{(r)}}(Ref:)
\DeclareComment{LATER}{LaterColor}{\EM{(L)}}(Later:)

\DeclareDocumentCommand{\DeclareCounter}{m}%
{%
 	\ifcsname c@#1\endcsname%
	\else%
		\newcounter{#1}%
	\fi%
}

\DeclareDocumentCommand{\MyQED}{}{\qed}

\DeclareDocumentEnvironment{proof}{d() D[]{\MyQED} d<>}
{
\noindent\IfNoValueTF{#1}
{\emph{Proof.\!\!}}
{\emph{Proof\ #1.\ }}
}
{
\IfValueTF{#3}
	{%
	\ExplSyntaxOff
	\RightJustify{\EM{#2_{#3}}}
	\vspace{1ex}
	}
	{
	\RightJustify{\EM{#2}}
	\vspace{1ex}
	}
}

\DeclareCounter{ProofLabelcOUntEr}
\DeclareDocumentCommand{\ProofLabel}{}{%
\addtocounter{ProofLabelcOUntEr}{1}
\label{cUrrEntProoflAbEl\arabic{ProofLabelcOUntEr}}
}

\DeclareDocumentCommand{\ProofRef}{D<>{1}}
{%
\ref{cUrrEntProoflAbEl\arabic{ProofcOUntEr#1}}
}
\DeclareDocumentCommand{\ProofCref}{D<>{1}}
{%
\cref{cUrrEntProoflAbEl\arabic{ProofcOUntEr#1}}
}

\DeclareDocumentEnvironment{Proof}{d() D[]{\MyQED} D<>{1}}
{
\DeclareCounter{ProofcOUntEr#3}%
\setcounter{ProofcOUntEr#3}{\value{ProofLabelcOUntEr}}%
\begin{proof}(#1)[#2]<\ProofRef<#3>>%
}
{
\end{proof}
}

\ExplSyntaxOn
\def\TheoremDepth{section}

\DeclareDocumentCommand{\DeclareTheorem}{m o m o}{%
\IfNoValueTF{#4}
	{%
	\IfNoValueTF{#2}
		{%
		\newtheorem{#1vArIAblE}{#3}
		}
		{%
		\newtheorem{#1vArIAblE}[#2vArIAblE]{#3}
		}
	}
	{%
	\newtheorem{#1vArIAblE}{#3}[#4]%
	}
\newtheorem*{#1vArIAblE*}{#3}

\DeclareDocumentEnvironment{#1}{o o}

	{
	\IfValueT{##2}%
		{
		\begin{spacing}{##2}
		}
	\IfValueTF{##1}
		{
		\begin{#1vArIAblE}[##1]
		}
		{
		\begin{#1vArIAblE}
		}
	\ProofLabel
	}
	{
	\IfValueT{##2}%
		{
		\end{spacing}{##2}
		}
	\end{#1vArIAblE}
	}

\DeclareDocumentEnvironment{#1*}{o o}

	{
	\IfValueT{##2}%
		{
		\begin{spacing}{##2}
		}
	\IfValueTF{##1}
		{
		\begin{#1vArIAblE*}[##1]
		}
		{
		\begin{#1vArIAblE*}
		}
	}
	{
	\IfValueT{##2}%
		{
		\end{spacing}{##2}
		}
	\end{#1vArIAblE*}
	}
}
\ExplSyntaxOff

\theoremstyle{plain}
\DeclareTheorem{theorem}{Theorem}[\TheoremDepth]
\DeclareTheorem{proposition}[theorem]{Proposition}
\DeclareTheorem{lemma}[theorem]{Lemma}
\DeclareTheorem{corollary}[theorem]{Corollary}
\DeclareTheorem{claim}[theorem]{Claim}
\DeclareTheorem{remark}[theorem]{Remark}
\DeclareTheorem{example}[theorem]{Example}

\theoremstyle{definition}
\DeclareTheorem{definition}[theorem]{Definition}
\DeclareTheorem{axiom}[theorem]{Axiom}
\DeclareTheorem{convention}[theorem]{Convention}
\DeclareTheorem{conjecture}[theorem]{Conjecture}

\theoremstyle{remark}
\DeclareTheorem{question}[theorem]{Question}

\begin{document}

\setlist[itemize]{itemsep=0.75em}

\title{Cohen Generic Structures with Functions}

\thanks{The second author's research has been supported by a grant from IPM (No. 1402030417).}

\begin{abstract}
Suppose $\LangBot\subseteq \Lang$ are languages where $\Lang \setminus \LangBot$ is relational. Additionally, let $\AgeK$ be a strong \Fraisse\ class in $\Lang$. We consider the partial ordering, under substructure, of those elements in $\AgeK$ whose reduct to $\LangBot$ are substructures of a fixed $\LangBot$-structure $\MBot$. In this paper, we establish that, under general conditions, this partial order satisfies the $|\MBot|$-chain condition.

Furthermore, under these conditions, we demonstrate that any generic for such a partial order satisfies the theory of the \Fraisse\ limit of $\AgeK$, provided $\MBot$ satisfies the theory of \Fraisse\ limit of its age. 

We also provide general conditions that guarantee all such generics to be rigid, as well as conditions ensuring that these generics possess large automorphism groups.
   
\end{abstract}

\author{Nathanael Ackerman}
\address{Harvard University,
Cambridge, MA 02138, USA}
\email{nate@aleph0.net}

\author{Mohammad Golshani}
\address{School of Mathematics, Institute for Research in Fundamental Sciences (IPM), P.O. Box:
19395-5746, Tehran, Iran.}
\email{golshani.m@gmail.com}

\author{Mostafa Mirabi}
\address{Department of Mathematics and Computer Science, The Taft School, Watertown, CT 06795, USA}
\email{mmirabi@wesleyan.edu}


\maketitle

\section{Introduction}

Given a strong \Fraisse\ class $\AgeK$ in a relational language $\Lang$, the concept of a Cohen generic structure for $\AgeK$ was introduced in \cite{Cohen} and \cite{Golshani}. Such a structure is obtained as a generic when $\AgeK$ is treated as a partial ordering whose order is the substructure relation. In \cite{Cohen} and \cite{Golshani}, they showed, among other things, that when $\Lang$ is a countable relational language, $\AgeK$ has only countably many elements up to isomorphism, and all elements of $\AgeK$ are finite, the partial ordering $\AgeK$ is ccc (countable chain condition) and hence preserves cardinals under forcing.

While the preservation of cardinals when forcing with strong \Fraisse\ classes in a relational language is an interesting phenomenon, the restriction to relational languages is a significant one. In particular, there are many strong \Fraisse\ classes which can be built on top of algebraic structures, and hence for which the results of \cite{Cohen} and \cite{Golshani} do not apply. 

The following are two examples to keep in mind when considering uncountable generic objects for a strong \Fraisse\ class. Both examples make fundamental use of the function symbols in the language.

\begin{example}
\label{Example: Colored vector spaces}
Suppose $\Field$ is a finite field. Let $\LangBot_{\Field} = \{+\} \cup \{f_c\}_{c \in \Field} \cup \{v_0\}$ be the one sorted language where $v_0$ is a constant, $+$ is a binary function symbol and for each $c \in \Field$, $f_c$ is a unary function symbol. Let $\AgeKBot_{\Field}$ be the age in $\LangBot_{\Field}$ whose structures are finite dimensional vector spaces over $F$ where $v_0$ is interpreted as the $0$ object, $f_c$ is interpreted as ``multiplication by $c$'' (for $c \in \Field$) and $+$ is vector addition. 

Let $\Lang_{\Field, n} = \LangBot_{\Field} \cup \{R\}$ where $R$ has arity $|\Field|^n$. We then let $\AgeK$ consist of $\Lang_{\Field, n}$ structures $\cA$ such that 
\begin{itemize}
\item $\cA \rest[\LangBot_{\Field}]\in \AgeKBot_{\Field}$, 

\item if $\cA \models R(a_0, \dots, a_{k-1})$ where $k = |\Field|^n$, then 
\begin{itemize}
\item $\{a_0, \dots, a_{k-1}\}\rest[\LangBot_{\Field}] \in \AgeKBot_{\Field}$,

\item for all $i < j < k$, $a_i \neq a_j$, 

\item if $\sigma$ is a permutation of $[k]$, then $\cA \models R(a_{\sigma(0)}, \dots, a_{\sigma(k-1)})$. 
\end{itemize}

\end{itemize}

Elements of $\AgeK$ then consist of vector spaces  with distinguished $n$-dimensional subspaces that satisfy the relation $R$. 

These can be thought of as vector space analogs of $n$-ary hypergraphs and the \Fraisse\ limit of $\AgeK$ can be thought of as the infinite dimensional vector space analog of the generic $n$-ary hypergraph. 
\end{example}

\begin{example}
\label{Example: Colored edges of trees}
Let $\LangBot_{\Tree} = \{p, r\}$ where $r$ is a constant and $p$ is a unary function symbol. Let $\AgeKBot_{\Tree}$ be the age over $\LangBot_{\Tree}$ consisting of those structures $\cA$ such that 
\begin{itemize}
\item $\cA \models p(r) = r$, 

\item $\cA \models (\forall a)\, p^{|\cA|}(a) = r$.
\end{itemize}

Elements of $\AgeKBot_{\Tree}$ can be thought of as trees with root $r$ and predecessor relation $p$. 

Let $\Lang_{\Tree} = \LangBot_{\Tree} \cup \{R\}$ where $R$ is a binary relation. Let $\AgeK_{\Tree}$ consist of those $\Lang_{\Tree}$-structures $\cA$ such that 
\begin{itemize}
\item $\cA \rest[\LangBot_{\Tree}] \in \AgeKBot_{\Tree}$, 

\item $\cA \models R(a, b) \rightarrow p(a) = b$.

\end{itemize}

Elements of $\AgeK_{\Tree}$ can be thought of as finite trees where for each node we either color the edge between the node and its predecessor (with the relation $R$) or we don't. 

A \Fraisse\ limit of $\AgeK_{\Tree}$ is then the countably branching tree where all edges of the tree are colored generically. 

When we move to case of uncountable models whose age is $\AgeKBot_{\Tree}$ though a phenomenon emerges which wasn't present in \cref{Example: Colored vector spaces}. Specifically for every uncountable cardinal $\kappa$ there are $2^\kappa$ many models of the \Fraisse\ theory of $\AgeKBot_{\Tree}$ of size $\kappa$. 

When defining generic structures we will want to fix a $\LangBot_{\Tree}$-structure $\MBot_{\Tree}$ and force with those elements of $\AgeK_{\Tree}$ whose reduct to $\LangBot_{\Tree}$ is a substructure of $\MBot_{\Tree}$. As such, the generic structure we end up with will depend on the choice of $\LangBot_{\Tree}$-structure we started with. 
\end{example}

One of the main results of \cite{Cohen} and \cite{Golshani} is that when one forces with a strong \Fraisse\ class cardinals are preserved. The restriction in \cite{Cohen} and \cite{Golshani} to relational languages is necessary for their proofs as they rely on the $\Delta$-system lemma for sets, and hence require that every subset of a structure gives rise to a substructure. 

In this paper, we will show how to weaken the assumption that the language is relational. For specificity we now describe our framework in the case where the language is countable, all elements of the age are finitely generated, and all there are (up to isomorphism) only countably many elements in the age.  Note both \cref{Example: Colored vector spaces} and \cref{Example: Colored edges of trees} satisfy these conditions.

Suppose $\LangBot \subseteq \Lang$ are languages $\Lang \setminus \LangBot$ is relational. Also suppose $\AgeK$ is a strong \Fraisse\ class in $\Lang$ where $\AgeK\rest[\LangBot]$ is a strong \Fraisse\ class in $\LangBot$. Suppose $\MBot$ is any $\LangBot$-structure which satisfies the theory of the \Fraisse\ limit of $\AgeK\rest[\LangBot]$. We consider the partial ordering whose objects are elements of $\AgeK$ whose reduct to $\LangBot$ are substructures of $\MBot$ and where the ordering is substructure. By first proving an analog of the $\Delta$-system lemma we will show that this partial ordering is ccc. In a similar approach to \cite{Cohen}, we will also consider properties of this forcing when the age consists of non-finitely generated objects. 

Both \cite{Cohen} and \cite{Golshani} also show  an interesting phenomenon occurs with many Cohen generic structures. For example, \cite{Cohen} shows that there are several natural classes where the \Fraisse\ limits have many automorphisms but the Cohen generic structures for the classes, when the underlying sets are uncountable, are always rigid. We will show that this phenomenon extends to the case when there are function symbols in $\LangBot$. Specifically, we will give conditions which will ensure that the Cohen generic structure is rigid. We will also give some conditions which ensure the Cohen generic structure has a large set of automorphisms. 

\subsection{Notation}
We let $\ORD$ be the collection of ordinals. We let $\cofinal(\kappa)$ be the \emph{cofinality} of a cardinal $\kappa$. We let $\tc(X)$ be the \emph{transitive closure} of the set $X$. If $X$ is a set we let $\Perm(X)$ be the collection of \emph{permutations} of $X$. 
A \defn{$\Delta$-system} is a collection of sets $\{S_i\}_{i \in \kappa}$  for which there is a set $S$ such that for all $i < j\in \kappa$, $S_i \cap S_j = S$.

By a language we mean a quadruple $\Lang = \{\Sort[\Lang], \Rel[\Lang], \Func[\Lang], \ar[\Lang]\}$ where $\Sort[\Lang]$ is the collection of sorts, $\Rel[\Lang]$ is the collection of relation symbols, $\Func[\Lang]$ is the collection of function symbols and $\ar[\Lang]\:\Rel[\Lang] \cup \Func[\Lang] \to \bigcup_{n \geq 1}\Sort[\Lang]^{n}$ is the function which takes a symbol and returns its arity. If $f \in \Func[\Lang]$ and $\ar[\Lang](f) = (\<S_i\>_{i \in [k]}, S_{k})$ then we say $f$ is a function symbol from $\<S_i\>_{i \in [k]}$ to $S_k$ and write $f\:\prod_{i \in [k]} S_i \to S_k$. The cardinality of a language is the cardinality of $\Sort[\Lang] \cup \Rel[\Lang] \cup \Func[\Lang]$. When $|\Sort[\Lang]| = 1$ we will omit mention of it. 

If $\cM, \cN$ are $\Lang$-structure we write $\cM \subseteq \cN$ to denote that $\cM$ is a substructure of $\cN$. We will let $\UnderSet{\cM}$ denote the underlying set of $\cM$. If $\aa$ is a tuple we write $\aa \in \cM$ if every component of $\aa$ is an element of $\cM$, i.e. $\aa \in \UnderSet{\cM}^{|\aa|}$ where $|\aa|$ is the length of the tuple $\aa$.  If $\LangBot \subseteq \Lang$ we let $\cM\rest[\LangBot]$ be the $\LangBot$-structure which is the restriction of $\cM$ to $\LangBot$. 

We define the \defn{term definable closure} of a set $A \subseteq \cM$ to be the smallest subset of $\cM$ containing $A$ and closed under taking terms. We denote it by $\tdcl[\cM](A)$. If $\aa$ is a tuple which enumerates $A$, we then let $\tdcl[\cM](\aa) = \tdcl[\cM](A)$.  We say $\cM$ is \defn{generated} by a set or tuple $X$ if $\tdcl[\cM](X) = \cM$. We say $\cM$ is \defn{$< \lambda$-generated} if there is a tuple of size $< \lambda$ which generates $\cM$. We say $\cM$ is \defn{finitely generated} if it is $< \w$-generated. 

Suppose $\cM$ is an $\Lang$-structure and $\aa \in \cM$ is a tuple. We let $\qftp[\cM](\aa)(\xx)$ be the \defn{quantifier-free type} of $\aa$ in $\cM$ with variables $\xx$, i.e. the collection of all literals in $\xx$ true about tuples of $\aa$ in $\cM$. The following folklore result is immediate but important.   

\begin{lemma}
Suppose $\cM, \cN$ are $\Lang$-structures where $\aa$ generates $\cM$ and $\bb$ generates $\cN$. If $\qftp[\cM](\aa)(\xx) = \qftp[\cN](\bb)(\xx)$ then $\cM \cong \cN$. 
\end{lemma}

We say a subset $T \subseteq \Liw(\Lang)$ is \defn{complete} if for every $\varphi \in \Liw(\Lang)$, either $T \vdash \varphi$ or $T \vdash \neg \varphi$ where $\vdash$ is the standard provability relation for $\Liw(\Lang)$ (see \cite{Barwise}). Two $\Lang$-structures $\cM$ and $\cN$ are \defn{potentially isomorphic} if they satisfy the same sentences of $\Liw(\Lang)$.
It can be shown that $\cM$ and $\cN$ are potentially isomorphic if and only if they are isomorphic in some forcing extension of the universe (see \cite{Marker2016-MARLOI}).  


  
We will use $V$ to denote a model of Zermelo-Fraenkel set theory with the axiom of choice. For definitions with respect to forcing we refer the reader to \cite{MR756630}. In particular, we will use the convention of \cite{MR756630} that if $(P, \leq)$ is a partial ordering  and $p \leq q$ then $p$ is \emph{stronger} than $q$.  
\section{\Fraisse\ classes}

We now review basic notions of \Fraisse\ theory. See, for example, \cite[Section 7.1]{MR1221741} for more on countable \Fraisse\ theory.

\begin{definition}
An \defn{age} in a language $\Lang$ is a collection of $\Lang$-structures closed under isomorphism. If $\AgeK$ is an age in $\Lang$ then we define the following.  
\begin{itemize}
\item For a cardinal $\kappa$ we let $\AgeK<\kappa>$ be the collection of elements of $\AgeK$ which are generated by a set of size $< \kappa$. 

\item For $\LangBot \subseteq \Lang$ let $\AgeK\rest[\LangBot] = \{\cA \rest[\LangBot] \st \cA \in \AgeK\}$. 

\end{itemize} 

We say $\AgeK$ is \defn{$\kappa$-bounded} if $\AgeK = \AgeK_\kappa$. 

If $\cM$ is a $\Lang$-structure we define the \defn{$\kappa$-generated age} of $\cM$ to be the collection of $\Lang$-structures isomorphic to a substructure of $\cM$ generated by a set of size $<\kappa$. 
\end{definition} 

\begin{definition}
Suppose $\AgeK$ is an age. 
\begin{itemize}
\item We say $\AgeK$ has the \defn{hereditary property (HP)} if whenever $\cA_1 \in \AgeK$ and $\cA_0 \subseteq \cA_1$, then $\cA_0 \in \AgeK$. 

\item We say $\AgeK$ has the \defn{joint embedding property (JEP)} if whenever $\cA_0, \cA_1 \in \AgeK$ there is an $\cA_2\in \AgeK$ and embeddings $i\:\cA_0 \to \cA_2$ and $j\:\cA_1\to \cA_2$. 

\item Suppose $i_1\:\cA_0 \to \cA_1$, $i_2\:\cA_0 \to \cA_2$ are embeddings with $\cA_0, \cA_1, \cA_2 \in \AgeK$. An \defn{amalgam} of $(i_1, i_2)$ is a pair of embeddings $(j_1, j_2)$ such that, for some $\cA_3 \in \AgeK$, $j_1\:\cA_1\to \cA_3$,  $j_2\:\cA_2 \to \cA_3$ and $j_1 \circ i_1 = j_2 \circ i_2$. If further $j_1``[\cA_1] \cap j_2``[\cA_2] = (j_1 \circ i_1)``[\cA_0]$, then we say $(j_1, j_2)$ is a \defn{disjoint amalgam} of $(i_1, i_2)$.

\item We say $\AgeK$ has the \defn{amalgamation property (AP)} if whenever $\cA_0, \cA_1, \cA_2 \in \AgeK$ and $i_1\:\cA_0 \to \cA_1$, $i_2\:\cA_0 \to \cA_2$ are embeddings, then $(i_1, i_2)$ has an amalgam.

\item We say $\AgeK$ has the \defn{strong amalgamation property (SAP)} if each such $(i_1, i_2)$ has a disjoint amalgam.

\item We say $\AgeK$ is \defn{closed under less than $\kappa$-unions (C$\kappa$U)} if for all $\gamma < \kappa$ and $(\cA_i)_{i \in \gamma} \subseteq \AgeK$ with $\cA_i \subseteq \cA_j$ for $i < j \in \gamma$ we have $\bigcup_{i \in \gamma} \cA_i \in \AgeK$. 
\end{itemize}

A \defn{\Fraisse\ class} is an age which satisfies (HP), (JEP) and (AP). A \defn{strong \Fraisse\ class} is a \Fraisse\ class which satisfies (SAP). 
\end{definition}

\begin{definition}
\label{Definition Fraisse limit}
Suppose $\AgeK$ is a \Fraisse\ class in $\Lang$. A \defn{\Fraisse\ limit} of $\AgeK$ is an $\Lang$ structure $\cM$ such that 
\begin{itemize}
\item[(a)] For every $\cB \in \AgeK$ there is an embedding from $\cB$ into $\cM$, 

\item[(b)] If $\cB \in \AgeK$ and $i, j\:\cB \to \cM$ are embeddings, then there is a $g \in \Aut(\cM)$ such that $i = g \circ j$. 

\end{itemize}
\end{definition}

The following definition will be useful. 
\begin{definition}
Suppose $\cM$ is an $\Lang$-structure and $A \subseteq \cM$. We define the \defn{group theoretic definable closure} of $A$ to be 
\[
\gdcl[\cM](A) = \{b \in \cM \st (\forall g \in \Aut(\cM))\, g\rest[A] = \id \rightarrow g(b) = b\}.
\]
We say $\cM$ has \defn{trivial group theoretic definable closure} if for all finitely generated substructures $\cA\subseteq \cM$,\, $\gdcl[\cM](\UnderSet{\cA}) = \UnderSet{\cA}$. 
\end{definition}

Note the following is a well known result (see \cite[Ch. 7.1]{MR1221741})
\begin{lemma}
\label{Fraisse limits exist and trivial group theoretic dcl}   
Suppose $\Lang$ is a countable language, $\AgeK$ is an $\w$-bounded \Fraisse\ class in $\Lang$ and $\AgeK$ has only countably many elements up to isomorphism. Then there is a countable \Fraisse\ limit $\cM$ of $\AgeK$ which is unique up to isomorphism. 

Furthermore, if $\Lang$ is relational, then $\AgeK$ is a strong \Fraisse\ class if and only if $\cM$ has trivial group theoretic definable closure. 
\end{lemma}

\begin{definition}
\label{Definition Fraisse Theory}
Suppose $\AgeK$ is a \Fraisse\ class which is $\kappa$-bounded. We let $\ThFraisse_\kappa(\AgeK) \subseteq \Lik(\Lang)$ be the theory with the following axioms. 
\begin{itemize}
\item[(a)] for all $\cA \in \AgeK$ and $\aa \in \cA$ with $\aa$ generating $\cA$
\[
(\exists \xx)\, \bigwedge\qftp[\cA](\aa)(\xx). 
\]

\item[(b)] for all $\cA, \cB \in \AgeK$ with $\cA\subseteq \cB$, $\aa$ generating $\cM$ and $\bb$ generating $\cB$
\[
(\forall \xx)\, \bigg[ \bigwedge\qftp[\cA](\aa)(\xx) \rightarrow (\exists \yy)\, \bigwedge\qftp[\cB](\aa\bb)(\xx\yy)\bigg]. 
\]
\end{itemize}
We call $\ThFraisse_\kappa(\AgeK)$ the \defn{$\kappa$-\Fraisse\ theory} of $\AgeK$. 
\end{definition}

Note that if $i\:\cA_0 \cong \cA_1$ and $\aa \in \cA_0$, then $\qftp[\cA_0](\aa)(\xx) = \qftp[\cA_1](i(\aa))(\xx)$. Therefore, as every $\kappa$-bounded \Fraisse\ class has only a set size collection of objects up to isomorphism, $\bigwedge \ThFraisse_\kappa(\AgeK_\kappa) \in \Lik(\Lang)$ for all $\kappa$.

The following is folklore. 
\begin{lemma}
\label{Fraisse limits satisfy Fraisse theories}
If $\AgeK$ is $\kappa$-bounded and $\cM$ is a \Fraisse\ limit of $\AgeK$, then $\cM \models \ThFraisse_\kappa(\AgeK)$. 
\end{lemma}
\begin{proof}
That $\cM$ satisfies the axioms in \cref{Definition Fraisse Theory} (a) follows from \cref{Definition Fraisse limit} (a). Now suppose $\cA, \cB, \aa, \bb$ are as in \cref{Definition Fraisse Theory} (b) and suppose $\cM \models \bigwedge\qftp[\cA](\aa)(\cc)$ for some tuple $\cc \in \cM$. There is then an embedding $i\:\cB \to \cM$ by \cref{Definition Fraisse limit} (a). But there is also an embedding $j\:\cc \to i(\aa)$ as $\aa$ has the same quantifier-free type in $\cA$ as $\cc$ does in $\cM$. Therefore, by  \cref{Definition Fraisse limit} (b) there is an automorphism $g$ of $\cM$ taking $i(\aa)$ to $\cc$. Let $\dd = g(i(\bb))$. But then $\aa\bb$ has the same quantifier-free type in $\cB$ as $i(\aa\bb)$ does in $\cM$, which in turn has the same quantifier-free type in $\cM$ as $g(i(\aa\bb)) = \cc \dd$. Therefore $\cM$ satisfies the axioms in \cref{Definition Fraisse Theory} (b).
\end{proof}

The following is folklore.
\begin{proposition}
Suppose $\AgeK$ is $\w$-bounded. Then $\ThFraisse_\w(\AgeK)$ is complete for $\Liw(\Lang)$. 
\end{proposition} 
\begin{proof}
As $\ThFraisse_\w(\AgeK)$ is a set, it is complete for $\Liw(\Lang)$ if and only if $\bigwedge \ThFraisse_\w(\AgeK)$ is a Scott sentence in any model of set theory where it is countable. We can therefore assume without loss of generality that $\ThFraisse_\w(\AgeK)$ is countable.

In particular, we have $\AgeK$ has only countably many elements up to isomorphism. Therefore by \cref{Fraisse limits exist and trivial group theoretic dcl} $\AgeK$ has a countable \Fraisse\ limit $\cM$. But, by \cref{Fraisse limits satisfy Fraisse theories} we have $\cM \models \ThFraisse_\w(\AgeK)$.

Finally, a straightforward back and forth argument shows that any two countable models of $\ThFraisse_\w(\AgeK)$ are isomorphic. Therefore $\bigwedge \ThFraisse_\w(\AgeK)$ is a Scott sentence of $\cM$ and hence complete.
\end{proof}

In our context, we want to consider a pair of languages $\LangBot \subseteq \Lang$ where $\Lang\setminus \LangBot$ is relational. We then consider a strong \Fraisse\ class $\AgeK$ in $\Lang$ such that $\AgeK\rest[\LangBot]$ is a strong \Fraisse\ class in $\LangBot$. We then fix some large $\LangBot$-structure $\MBot$ and we will want to force with elements of $\AgeK$ whose restriction to $\LangBot$ is a substructure of $\MBot$. The following notation will be convenient. 

\begin{definition}
Suppose $\LangBot \subseteq \Lang$ are languages, $\MBot$ is a $\LangBot$-structure and $\AgeK$ is an age over $\Lang$. We let 
\[
\AgeK[\MBot] = \{\cA \in \AgeK \st \cA\rest[\LangBot] \subseteq \MBot\}. 
\]
\end{definition}

There is though one issue with the above framework. Even though we have strong amalgamation in both $\AgeK$ and $\AgeK\rest[\LangBot]$ it is possible that given $\Lang$-structures $\cA \subseteq \cB$ and $\cA \subseteq \cC$ in $\AgeK[\MBot]$ the amalgamation of $\cB\rest[\LangBot]$ and $\cC\rest[\LangBot]$ over $\cA\rest[\LangBot]$, which is imposed by $\MBot$, is not compatible with any amalgamation of $\cB$ and $\cC$ over $\cA$. We therefore will need a stronger version of amalgamation which ensures every disjoint amalgam in $\AgeK\rest[\LangBot]$ can be expanded to a disjoint amalgam in $\AgeK$. We now make this precise. 

\begin{definition}
Suppose 
\begin{itemize}
\item $\LangBot \subseteq \Lang$, 

\item $\AgeK$ is a strong \Fraisse\ class in $\Lang$ with $\AgeK\rest[\LangBot]$ a strong \Fraisse\ class in $\LangBot$. 
\end{itemize}
We say $\AgeK$ has the  \defn{extended strong amalgamation property over $\LangBot$} if, whenever
\begin{itemize}
\item $\cA, \cB, \cC \in \AgeK$ with $\cA \subseteq \cB$ and $\cA \subseteq \cC$ and $\UnderSet{\cB} \cap \UnderSet{\cC} = \UnderSet{\cA}$ 

\item $\DBot \in \AgeK\rest[\LangBot]$ with $\cB\rest[\LangBot] \subseteq \DBot$, $\cC\rest[\LangBot] \subseteq \DBot$ and where $\DBot$ is generated by $\UnderSet{\cB} \cup \UnderSet{\cC}$, 
\end{itemize}
Then there is a $\cD \in \AgeK$ such that $\cB \subseteq \cD$, $\cC \subseteq \cD$ and $\cD \rest[\LangBot] = \DBot$. 

\end{definition}

\section{Forcings}

We now introduce the forcing notions we are interested in. 

\begin{definition}
Suppose 
\begin{itemize}
\item $\LangBot \subseteq \Lang$, 

\item $\Lang \setminus \LangBot$ is relational, 

\item both $\AgeK$ and $\AgeK\rest[\LangBot]$ are strong \Fraisse\ classes, 

\item $\MBot \models \ThFraisse_\kappa(\AgeK\rest[\LangBot])$, 
\end{itemize}
We then let $\Fn<\MBot>[\AgeK](\kappa)$ be the partial ordering where 
\begin{itemize}
\item the elements are structures in $\AgeK<\kappa>[\MBot]$

\item for $\cM, \cN \in \Fn<\MBot>[\AgeK](\kappa)$, $\cM \leq \cN$ if $\cN \subseteq \cM$.
\end{itemize}

\end{definition}

Note for $\cM, \cN \in \Fn<\MBot>[\AgeK](\kappa)$ we have $\cM$ is stronger than $\cN$ precisely when $\cN$ is a substructure of $\cM$, i.e. $\cN \subseteq \cM$. 

\begin{proposition}
\label{propclosure}
Suppose $\lambda \leq \cofinal(\kappa)$ and $\AgeK$ is a strong \Fraisse\ class with (C$\lambda$U). Then $\Fn<\MBot>[\AgeK](\kappa)$ is $\lambda$-closed. 
\end{proposition}
\begin{proof}
Suppose $\<\cA_i\>_{i \in \gamma}$ is a decreasing sequence in $\Fn<\MBot>[\AgeK](\kappa)$ with $\gamma < \lambda$, i.e. for all $i \in \gamma$, $\cA_i$ is $<\kappa$-generated and for $i < j < \gamma$, $\cA_i \subseteq \cA_j$. Then, because $\AgeK$ has (C$\lambda$U), we know that $\bigcup_{i \in \gamma} \cA_i \in \AgeK$. Therefore, as $\gamma < \lambda \leq \cofinal(\kappa)$ we $\bigcup_{i \in \gamma} \cA_i \in \AgeK$ is $<\kappa$-generated. But then for all $j < \gamma$, $\cA_j \subseteq \bigcup_{i \in \gamma} \cA_i$. Therefore, as $\<\cA_i\>_{i \in \gamma}$ was arbitrary, $\Fn<\MBot>[\AgeK](\kappa)$ is $\lambda$-closed. 
\end{proof}

Our main result will give conditions when $\Fn<\MBot>[\AgeK](\kappa)$ isn't just $\kappa$-closed but has the $\kappa^+$-chain condition. First though we will need to prove an analog of the $\Delta$-systems lemma  which will play a crucial role in our result. For a proof of the $\Delta$-system lemma see, for example, \cite[Theorem II.1.5]{MR756630}.

\begin{lemma}[Term $\Delta$-Systems Lemma]
\label{Delta-Systems Lemma}
Suppose 
\begin{itemize}
\item $\kappa^{<\lambda} = \kappa$

\item $|\Lang| \leq \kappa$, 

\item $\cM$ is an $\Lang$-structure, 
 
\item $B$ is a collection of distinct $<\lambda$-generated substructures of $\cM$ with $|B| = \kappa^+$. 

\end{itemize}

Then there is a $\cN \subseteq \cM$ and a collection $(\cN_i)_{i \in \kappa^+}$ with $\cN_i \in B$ for all $i \in\kappa^+$ and such that $(\forall i < j \in \kappa^+)\, \cN_i \cap \cN_j = \cN$ 

\end{lemma}
\begin{proof}
First, note that if $\cM^\circ$ is the smallest substructure of $\cM$ generated by $\bigcup B$ then $|\cM^\circ| \leq |B| \cdot (< \lambda) \cdot |\Lang|= \kappa^+$. We can therefore assume without loss of generality that $|\cM| = \kappa^+$ and therefore we can assume without loss of generality that the underlying set of $\cM$ is the cardinal $\kappa^+$.  

Let $g\:B \to \cM^{<\lambda}$ be a map such that for all $b \in B$, $b$ is the substructure of $\cM$ generated by $g(b)$. 
%
%
Let $X  = \{\Lang, \Rel[\Lang], \Func[\Lang], \ar[\Lang]\} \cup \{B\} \cup \{g\} \cup \{\cM\}$. 

As $\kappa^{<\lambda} = \kappa$ we have $\lambda \leq \kappa$.
Let $\theta > \kappa$ be a large enough regular cardinal, let $\mathcal{H}(\theta)$ be the class of hereditarily sets of cardinality less than $\theta$, and let
$<^*$ be a well-ordering of $\mathcal{H}(\theta)$. Consider the structure $\mathcal{H}=(\mathcal{H}(\theta), \in, <^*).$
Let $(W_i)_{i \in \kappa+1}$ be an increasing and continuous sequence of elementary substructures of $\mathcal{H}$ such that 
\begin{itemize} 
\item[(a)] $X \in W_0$, 

\item[(b)] $|W_i| = \kappa$ for all $i \in \kappa$, 

\item[(c)] $W_i^{<\lambda} \subseteq W_{i+1}$ for all $i \in \kappa$. 

\item[(d)] if $x \in W_i \cap \ORD$ and $y \in \kappa^+$ with $y \in x$, then $y \in W_{i+1}$, 

\item[(e)] $W_{\gamma} = \bigcup_{i \in \gamma} W_{i}$, for limit ordinal $\gamma$.
\end{itemize}
Note we can always find such a sequence as $\kappa^{<\lambda} = \kappa$. 

By (d) we have $(\kappa^+  \cap W_\kappa) \in \ORD$. Let this value be $\zeta$. As $|W_\kappa| = \kappa$ we have $\zeta < \kappa^+$. Therefore there must be some $B_0 \in B$ such that $B_0 \not \in W_\kappa$. Let $r = g(B_0) \cap \zeta$ and let $r^+$ be the substructure of $\cM$ generated by $r$. Note $r$ and hence $r^+$ are in $W_\kappa$. Note that if $<_{\cM}$ is the ordering on $\cM$ given from the fact that the underlying set of $\cM$ is $\kappa^+$, then $(\forall x \in B_0 \setminus r^+)\, x \not \in W_\kappa$ and hence $x >_{\cM} \zeta$.

Suppose $\alpha \in \zeta$. We then have 
\[
\mathcal{H} \models (\exists B_0 \in B)\, \Bigl( r \subseteq \UnderSet{B_0} \,\And \,(\forall b \in \UnderSet{B_0} \setminus r^+)\, b >_{\cM} \alpha\Bigr). 
\]
Therefore we have for each $\alpha \in \zeta$ 
\[
W_\kappa \models (\exists B_0 \in B)\, \Bigl(r \subseteq \UnderSet{B_0} \, \And \, (\forall b \in \UnderSet{B_0} \setminus r^+)\, b >_{\cM} \alpha\Bigr).
\]
But then we also have 
\[
W_\kappa \models (\forall \alpha)(\exists B_0 \in B)\, \Bigl(r \subseteq \UnderSet{B_0} \,\And\, (\forall b \in \UnderSet{B_0} \setminus r^+)\, b >_{\cM} \alpha\Bigr).
\]
and so 
\[
\mathcal{H} \models (\forall \alpha)(\exists B_0 \in B)\,\Bigl( r \subseteq \UnderSet{B_0} \,\And\, (\forall b \in \UnderSet{B_0} \setminus r^+)\, b >_{\cM} \alpha\Bigr).
\]
But as $|b| \leq \kappa$ for all $b \in B$ we can find a sequence $(b_i)_{i \in \kappa^+}$ with $b_i \in B$ for all $i \in \kappa^+$ and such that 
\begin{itemize}
\item for all $i \in \kappa^+$, $r^+ \subseteq \UnderSet{b_i}$, 

\item for all $i < j \in \kappa^+$, $\max (\UnderSet{b_i}\setminus r^+) < \min (\UnderSet{b_j} \setminus r^+)$. 

\end{itemize}

Therefore $(b_i)_{i \in \kappa^+}$ is a $\Delta$-system. 
\end{proof}

Note that even in the case when $\kappa = \lambda = \w$ this result need not follow from the ordinary $\Delta$-systems lemma. This is because, in this case, the $\Delta$-systems lemma only deals with finite subsets of $\w_1$. However, in \cref{Delta-Systems Lemma} the elements of $B$ might be countable when there is a countable language. In our context though this isn't a problem as each of our sets is finitely generated and there cannot be a finitely generated substructure of $\cM$ which is a subset of $W_\w$ but is not also contained in $W_\w$.

We now prove our main theorem. 
\begin{theorem}
\label{Conditions which guarantee the chain condition}
Suppose 
\begin{itemize}

\item $\kappa^{<\lambda} = \kappa$, 

\item $\LangBot \subseteq \Lang$, $\Lang\setminus \LangBot$ is relational, and $|\Lang| \leq \kappa$, 

\item $\AgeK$ is a $\lambda$-bounded strong \Fraisse\ class with (C$\lambda$U), 

\item $\AgeK\rest[\LangBot]$ is a strong \Fraisse\ class, 

\item $\AgeK$ has the extended strong amalgamation property over $\LangBot$,

\item $\MBot$ is an $\LangBot$-structure, and

\item for every $\LangBot$-structure $\cX \subseteq \MBot$ which is $< \lambda$-generated there are at most $\kappa$-many elements $\cN \in \AgeK$ with $\cN \rest[\LangBot] = \cX$.
\end{itemize}
Then any subset $Y$ of $\Fn<\MBot>[\AgeK](\lambda)$ of size $\kappa^+$ has a subset $Y_0 \subseteq Y$ of size $\kappa^+$ such that every two elements of $Y_0$ have a common lower bound. 

In particular, this implies $\Fn<\MBot>[\AgeK](\lambda)$ has the $\kappa^+$-chain condition and, if $\kappa = \w_1$, $\Fn<\MBot>[\AgeK](\lambda)$ satisfies the Knaster condition. 
\end{theorem}
\begin{proof}
Suppose, to get a contradiction, that $\{\cA_i\}_{i \in \kappa^+} \subseteq \Fn<\MBot>[\AgeK](\kappa)$ and for any $J \subseteq \kappa^+$ with $|J| = \kappa^+$ there are $x, y \in J$ such that $\cA_{x}$ and $\cA_y$ have no common extension in $\Fn<\MBot>[\AgeK](\lambda)$. 

Let $\mbf{A} = \{\cA_i\rest[\LangBot]\}_{i \in \kappa^+}$. For each $i \in \kappa^+$ we have $\cA_i$ is $< \lambda$-generated as $\AgeK$ is $\lambda$-bounded. Therefore there are at most $\kappa$ many elements $\cN \in \AgeK$ with $\cN \rest[\LangBot] = \cA_i\rest[\LangBot]$. Therefore  $\kappa^+ = |\{\cA_i\}_{i \in \kappa^+}| \leq |\mbf{A}| \cdot \kappa$ and hence $|\mbf{A}| = \kappa^+$. 

But by \cref{Delta-Systems Lemma}, we can find a subset $I_0 \subseteq \kappa^+$ with $|I_0| = \kappa^+$ and a $\cX \subseteq \MBot$ such that $(\forall i < j \in I_0)\, \UnderSet{\cA_i} \cap \UnderSet{\cA_j} = \UnderSet{\cX}$. But then $\cX \in \AgeK\rest[\LangBot]$ and so is $< \lambda$-generated. Therefore there are at most $\kappa$-many elements $\cN \in \AgeK$ with $\cN \rest[\LangBot] = \cX$. But then there must be a subset $I_1 \subseteq I_0$ such that $|I_1| = \kappa^+$ and $(\forall i \in I_1)\, (\cA_i\rest[\UnderSet{\cX}])\rest[\LangBot] = \cX$. 

For $i < j \in I_1$ let $\cB_{i, j}^-$ be the subset of $\MBot$ generated by $\UnderSet{\cA_i} \cup \UnderSet{\cA_j}$. As $\AgeK$ has the extended strong amalgamation property over $\LangBot$ there must be a $\cB_{i, j} \in \AgeK$ such that $\cA_i \subseteq \cB_{i, j}$, $\cA_j \subseteq \cB_{i, j}$ and $\cB_{i, j}\rest[\LangBot] = \cB_{i, j}^-$. Then $I_1$ contradicts our assumption. 
\end{proof}

We end with a straightforward observation about the theory of generics for these partial orders. 

\begin{proposition}
\label{Forcing satisfies Fraisse theory}
Suppose 
\begin{itemize}

\item $\LangBot \subseteq \Lang$, $\Lang \setminus \LangBot$ is relational, and $|\Lang| \leq \kappa$, 

\item $\AgeK$ and $\AgeK\rest[\LangBot]$ are strong \Fraisse\ classes, 

\item $\AgeK$ is $\lambda$-bounded and satisfies the $(C\lambda U)$,

\item $\AgeK$ has the extended strong amalgamation property over $\LangBot$,

\item $\MBot$ is an $\LangBot$-structure which satisfies the $\ThFraisse_\kappa(\AgeK\rest[\LangBot])$,  

\item $G$ is generic for $\Fn<\MBot>[\AgeK](\kappa)$ over $V$ and $\cG = \bigcup G$. 
\end{itemize}
Then $V$ and $V[G]$ have the same sequences of elements of $V$ of length less than $\lambda$ and
 $V[G] \models \text{``}\cG \models \ThFraisse_\kappa(\AgeK)$''.
\end{proposition}
\begin{proof}
As $\AgeK$ satisfies the $(C\lambda U)$, by Proposition \ref{propclosure}, the forcing notion
$\Fn<\MBot>[\AgeK](\kappa)$ is $\lambda$-closed, and hence $V$ and $V[G]$ have the same sequences of elements of $V$ of length less than $\lambda$. 

For each $\cA \in \AgeK$ it is immediate that 
\[
D_{\cA} = \{\cB \in \AgeK[\MBot]\st \cA \subseteq \cB\}
\]
is dense. Therefore for all $\cA \in \AgeK$ and $\aa \in \cA$ with $\aa$ generating $\cA$, $\cG \models (\exists \xx)\, \bigwedge\qftp[\cA](\aa)(\xx)$.

Further, as $\AgeK$ has the extended strong amalgamation property over $\LangBot$ we have that for any $\cA, \cB \in \AgeK[\MBot]$ with $\cA \subseteq \cB$ we have $D_{\cA, \cB}$ is dense, where
$D_{\cA, \cB}=D^1_{\cA, \cB} \cup D^2_{\cA, \cB}$, and

\[
D^1_{\cA, \cB} = \{\cC \in \AgeK[\MBot] \st \cA \subseteq \cC \text{ and }(\exists \cB^* \in \AgeK[\MBot])\, \cA \subseteq \cB^* \subseteq \cC\text{ and }(\cA, \cB) \cong (\cA, \cB^*)\}
\]
and
\[
D^2_{\cA, \cB} = \{\cC \in \AgeK[\MBot] \st \cA\rest[\LangBot] \subseteq \cC\rest[\LangBot] \text{ and }\cA \not \subseteq \cC\}.
\]

Therefore if $\aa$ generates $\cA$ and $\bb$ generates $\cB$
\[
\cG \models (\forall \xx) \bigg[\bigwedge\qftp[\cA](\aa)(\xx) \rightarrow (\exists \yy)\, \bigwedge\qftp[\cB](\aa\bb)(\xx\yy)\bigg]. 
\]
But,  $V$ and $V[G]$ have the same sequences of elements of $V$ of length less than $\lambda$, and so in $V[G]$
\[
\cG \models (\forall \xx)\, \bigg[\bigwedge\qftp[\cA](\aa)(\xx) \rightarrow (\exists \yy)\, \bigwedge\qftp[\cB](\aa\bb)(\xx\yy)\bigg]. 
\]
\end{proof}

\section{Rigidity}  

There are many situations when the generics generated by forcing with an age give rise to rigid structures, even when the underlying functional structure is not rigid. In this section, we give general conditions which will ensure such generic structures are rigid. 

The conditions below which ensure rigidity of the generic are an abstraction of the techniques in \cite{Cohen} used to show rigidity of the uncountable generics for graphs, partial orders and linear orders. As we will see, the techniques in \cite{Cohen} make fundamental use of the fact that the languages are binary. 

In order to understand the motivation for the following technical definitions it is worth looking at a high level summary of the proof of rigidity, which occurs in several stages. Recall the set up is that we have languages $\LangBot \subseteq \Lang$ where $\Lang \setminus \LangBot$ is relational. We also have a strong \Fraisse\ class $\AgeK$ in $\Lang$ which satisfies the extended strong amalgamation property over $\LangBot$. Finally we have an uncountable $\LangBot$-structure $\MBot$ whose age is $\AgeK\rest[\LangBot]$. 

The proof occurs in several stages. First, consider pairs $(\cF, \cS)$ such that (a) $\cS$ is a sort of $\LangBot$ and where no non-identity automorphism of $\MBot$ is the identity on $\cS^{\MBot}$ and (b) $\cF$ is a subset $\MBot$ such that for two elements $x, y \in \cS^{\MBot}$ and any generic structure $\cG$ there is a tuple $\zz \in \cF$ such that $\qftp[\cG](x\zz) \neq \qftp[\cG](y\zz)$. We say such an $\cF$ has the \emph{splitting property} over $\AgeK$ and $\cS$ as we can always split the types of elements of $\cS$ using elements of $\cF$.

For such a pair any non-trivial automorphism $\cG$ does not fix $\cF$.  Specifically, if $h$ was a non-trivial automorphism of $\cG$ that fixed $\cF$, then there must be distinct  $x,y \in \cS^{\MBot}$ with $h(x) = y$. But then there would a $\zz \in \cF$ with $\qftp[\cG](x\zz) \neq \qftp[\cG](y\zz) = \qftp[\cG](h(x)h(\zz))$ getting us a contradiction. Our first assumption will be that there are many such pairs $(\cF_i, \cS_i)$ such that for distinct $i, j$, $\cF_i \cap \cF_j$ is the term closure of the empty set in $\MBot$.  

Now assume to get a contradiction that $\name{h}$ is the name for an automorphism in the generic extension $V[G]$. For any such $\cF$ we can find an $s_{\cF} \in\cF$ and a $\cB_{\cF}$ in our generic such that $\cB_{\cF}$ forces $\name{h}(s_{\cF}) \neq s_{\cF}$. As there are many such $\cF$ we can find a large collection such that all of $\cB_{\cF}$ are isomorphic. 

Using the analog of the $\Delta$-system lemma, we can therefore find a $\Delta$-system of elements $\cB_{\cF}$ such that for any such $\cF_0, \cF_1$ there is an isomorphism $i\:\cB_{\cF_0} \to \cB_{\cF_1}$ where $i(s_{\cF_0}) = s_{\cF_1}$ and $i(\name{h}(s_{\cF_0})) = \name{h}(s_{\cF_1})$. 

Our second assumption about the age, which we call the \emph{amalgamation of distinct $2$-types property}, essentially this says that whenever we have $xy$ and $x'y'$ of the same quantifier-free type we can find an amalgamation such that $\qftp(xx')  \neq \qftp(yy')$. By construction $\qftp[\cB_{\cF_0}](s_{\cF_0}\name{h}(s_{\cF_0})) = \qftp[\cB_{\cF_1}](s_{\cF_1}\name{h}(s_{\cF_1}))$. Therefore 
our second assumption will allow us to find an amalgamation $\cC$ of $\cB_{\cF_0}$ and $\cB_{\cF_1}$ with $\qftp(s_{\cF_0} s_{\cF_1}) \neq \qftp(\name{h}(s_{\cF_0}) \name{h}(s_{\cF_1}))$. But any such $\cC$ must forces that $\name{h}$ is not an automorphism, contradicting the fact that $\cB_{\cF_0}$ forces that $\name{h}$ is an automorphism and $\cB_{\cF_0} \subseteq \cC$.  

It is worth noting that in this argument we are focused on the quantifier-free type of pairs of elements. It is for this reason that in \cite{Cohen} all the examples of rigid structures where in binary languages. In  particular, the techniques do not apply, for example, to the case of a language $\Lang$ with a single $n$-ary relation ($n > 2$) and the age of all finite $\Lang$-structure. It is though natural to ask if an uncountable generic for such a structure must be rigid? We will show by expanding our structure to add names for all $n$-tuples, something which fundamentally requires functions, that our techniques prove such generics are rigid. 

We now turn to the definitions we will need for our main result of this section. 
\begin{definition}
\label{Amalgamation of distinct types property}

For a sort $\cS \in \LangBot$ we say an age $\AgeK$ in $\Lang$ has the \defn{amalgamation of distinct $2$-type property} (over $\LangBot$ and $\cS$) if whenever 
\begin{itemize}
   
\item[(a)] $\cA, \cB, \cC \in \AgeK$, 

\item[(b)] $\cA \subseteq \cB$, $\cA \subseteq \cC$, and $\UnderSet{\cB} \cap \UnderSet{\cC} = \UnderSet{\cA}$, 

\item[(c)] $\alpha\:\cB \to \cC$ is an isomorphism with $\alpha\rest[\cA] = \id_{\cA}$, 

\item[(d)] $x_0, x_1 \in \cS^{\cB} \setminus \cS^{\cA}$ are distinct elements with $\qftp[\cB](x_0) = \qftp[\cB](x_1)$,

\item[(e)] $\DBot \in \AgeK\rest[\LangBot]$ with $\cB\rest[\LangBot] \subseteq \DBot$ and $\cC\rest[\LangBot] \subseteq \DBot$, 
\end{itemize}
then there is a $\cD \in \AgeK$ such that 
\begin{itemize}
\item[(f)] $\cB \subseteq \cD$, $\cC \subseteq \cD$, 

\item[(g)] $\cD\rest[\LangBot] = \DBot$,

\item[(h)] $\qftp[\cD](x_0\alpha(x_0)) \neq \qftp[\cD](x_1\alpha(x_1))$
\end{itemize}

Note when there is only one sort we will omit mention of it. 
\end{definition}

It is worth noting that an age having the amalgamation of distinct $2$-types is not enough to ensure uncountable generics are rigid, at least when $\LangBot$ is not the empty language.  

\begin{example}    
\label{Example amalgamation of distict 2-types doesn't imply rigid generic}
Consider the language $\LangBot = \{U, W\}$ where $U$ and $W$ are unary relations. Let $\AgeKBot$ be the age consisting of all finite $\LangBot$-structures $\cA$ such that $\{U^{\cA}, W^{\cA}\}$ is a partition of the underlying set.  Let $\Lang = \LangBot \cup \{E\}$ and let $\AgeK$ be the collection of finite $\Lang$-structures such that if $E(x, y)$ holds, then $U(x) \leftrightarrow U(y)$. 

Now suppose $\cA, \cB, \cC, \DBot, x_0, x_1, \alpha$ are in \cref{Amalgamation of distinct types property}. As $\qftp[\cB](x_0) = \qftp[\cB](x_1)$ we know that either $\cB \models U(x_0) \And U(x_1)$ or $\cB \models W(x_0) \And W(x_1)$. In the first case, we then have $\DBot \models U(\alpha(x_0)) \And U(\alpha(x_1))$. But as $\alpha(x_0), \alpha(x_1) \not \in \cB$ there is a $\cD \in \AgeK$ with $\cD \rest[\LangBot] = \DBot$, $\cB, \cC \subseteq \cD$ and $\cD \models E(x_0, \alpha(x_0)) \And\neg E(x_1, \alpha(x_1))$.

But if $\cB \models W(x_0) \And W(x_1)$, then we can also find such a $\cD$ as the only restriction on the edge relation $E$ is that there are no edges between elements satisfying $U$ and elements satisfying $W$. In particular this proves that $\AgeK$ has the amalgamation of distinct $2$-type property over $\LangBot$.

Now suppose $\MBot_{\kappa, \w}$ is the unique $\LangBot$-structure (up to isomorphism) where $|U^{\MBot_{\kappa, \w}}| = \kappa$ and $|W^{\MBot_{\kappa, \w}}| = \w$. Suppose $\cM$ is any $\Lang$-structure whose age is contained in $\AgeK$. If $g_0\:\cM\rest[U^{\cM}] \to \cM\rest[U^{\cM}]$ is an automorphism and $g_1\:\cM\rest[W^{\cM}] \to \cM \rest[W^{\cM}]$ is an automorphism, then $g_0 \cup g_1\:\cM \to\cM$ is an automorphism. 
    
Let $G$ be a generic for $\Fn<\MBot_{\kappa, \w}>[\AgeK](\w)$ and let $\cG$ be the corresponding $\Lang$-structure. We then have that $\cR = (W^{\cG}, E^{\cG} \cap (W^{\cG} \times W^{\cG}))$ is isomorphic to the Rado graph. Let $I$ be the identity on $U^{\cG}$ and let $g$ be any automorphism of $\cR$ in $V[G]$. We then have that $I \cup g$ is an automorphism of $\cG$ and so $\cG$ is not rigid in $V[G]$.

\end{example}

In \cref{Example amalgamation of distict 2-types doesn't imply rigid generic} we were able to find generics which were not rigid because the underlying structure had a countable set whose structure was independent of the structure of the rest of the model. The next property which we will need guarantees that can't happen.

\begin{definition}
Suppose $\Lang$ is a language with sort $\cS$ and $\AgeK$ is an age over $\Lang$. We say $\cS$ \defn{pins} $\AgeK$ if whenever 
\begin{itemize}
\item $\cA, \cB \in \AgeK$, 

\item $i_0, i_1\:\cA \to \cB$ are isomorphisms, 

\item $i_0 \rest[\cS^{\cA}] = i_1 \rest[\cS^{\cA}]$ 
\end{itemize}
then $i_0 = i_1$. 
\end{definition}

In particular if $\cS$ pins $\AgeK$ and $\cM$ is any structure whose age is contained in $\AgeK$ then any non-trivial automorphism of $\cM$ must move an element of $\cS^\cM$. The following is an important example of this. 

\begin{example}
\label{Exaple of names for sequences} 
Let $\LangBot_{n} = \{U, W\} \cup \{f\} \cup \{\pi_i\}_{i \in [n]}$ where $U$ and $W$ are the sorts, $f\:U^n \to W$ and for each $i \in [n]$, $\pi_i\:W \to U$.  Let $\AgeKBot_{s, n}$ be the collection of structures such that 
\begin{itemize}  
\item $f$ is a bijection, 

\item for all $v \in W$, $f(\pi_0(v), \dots, \pi_{n-1}(v)) = v$. 
\end{itemize}

In this example we can think of elements of $W$ as \emph{names} for sequences of elements of $U$ of length $n$. 

If $\MBot$ is a $\LangBot$-structure with age $\AgeKBot_{s, n}$ then any non-trivial automorphism of $\MBot$ must be non-trivial on $W^{\MBot}$. Therefore $W$ pins $\AgeKBot_{s, n}$.    
\end{example}

Fundamental to the proof of the rigidity of generics is the fact that we can find many disjoint sets $\cF$ such that, whenever there is an automorphism $h$ of the structure, there is some element $s$ such that the quantifier-free type of $s$ over $\cF$ is different than the quantifier-free type of $h(s)$ over $\cF$. This will ensure that any non-trivial automorphism moves some element of $\cF$. We now make precise this property. 

\begin{definition}
\label{Secondary rigidity property}
Suppose $\cS$ pins $\AgeK$. We say $\cF \subseteq \MBot$ has the \defn{splitting property} (over $\AgeK$ and $\cS$ and $\MBot$) if whenever $\cA \in \AgeK[\MBot]$ and $x, y \in \cA^{\cS}$ there is a $\cB \in \AgeK[\MBot]$ with $\cA \subseteq \cB$ and a tuple $\zz \in \cB \cap \cF$ such that $\qftp[\cB](x \zz) \neq \qftp[\cB](y \zz)$.  
\end{definition}   

Note that if $\LangBot$ is the one sorted language with no relations and $\MBot$ is an infinite $\LangBot$-structure then every infinite subset $\cF\subseteq \MBot$ (of sufficient size) has the splitting property over any non-trivial age. As such this is a property which only is needed in the proofs when the language $\LangBot$ is non-trivial.

\begin{proposition}
\label{Main rigidity result}
Suppose $V$ is a model of $\ZFC$ in which the following hold. 
\begin{itemize}
\item $\kappa$ and $\lambda$ are cardinals with $\kappa^{<\lambda} = \kappa$ and $n \in \w$, 

\item $\LangBot \subseteq \Lang$ with $|\LangBot| \leq \kappa$ and $\Lang \setminus \LangBot$ relational. 

\item $\AgeK$ and $\AgeK\rest[\LangBot]$ are $\lambda$-bounded strong \Fraisse\ classes and $\AgeK$ has extended strong amalgamation over $\LangBot$, 

\item for all $\cA \in \AgeK$ there are, up to isomorphism which fix $\cA$, at most $\kappa$-many elements $\cB \in \AgeK$ with $\cA \subseteq \cB$.

\item $\AgeK\rest[\LangBot]$ is a strong \Fraisse\ class. 

\item $\MBot\models \ThFraisse_\lambda(\AgeK\rest[\LangBot])$
%

\item there is a collection of sorts $\{\cS_i\st i \in \kappa^+\}$, a collection of sorts $\{\cT_i\st i \in \kappa^+\}$, and a collection of sets  $\{\cF_i \st i \in \kappa^+\}$ such that for all $i \in \kappa^*$
\begin{itemize}

\item $\cF_i \subseteq \cS_i^{\MBot}$ and $\cF_i \in V$ 

\item $\AgeK$ has the amalgamation of distinct $2$-type property over $\LangBot$ and $\cS_i$, 

\item $\cT_i$ pins $\AgeK$

\item $\cF_i$ has splitting property over $\AgeK$ and $\cT_i$, 

\end{itemize}
and $(\forall i < j \in \kappa)\, \UnderSet{\cF_i} \cap \UnderSet{\cF_j} = \UnderSet{\tdcl[\MBot](\emptyset)}$

\end{itemize}
Then whenever $G$ is a generic for $\Fn<\MBot>[\AgeK\rest[\LangBot]](\lambda)$ over $V$, $\bigcup G$ is rigid in $V[G]$. 
\end{proposition} 
\begin{proof}
Let $\cG = \bigcup G$. Suppose, to get a contradiction, this result fails. Then there must be some name $\name{h}$ for $\Fn<\MBot>[\AgeK](\lambda)$ in $V$ and $\cA \in G$ such that 
\[
\cA \forces[\Fn<\MBot>[\AgeK](\lambda)] \name{h}\text{ is a non-trivial automorphism of }\cG.
\]  
%
%

We now show that for each $\cF_i \in V$ the automorphism given by $\name{h}^G$ is not the identity on those $\cF_i$.
\begin{claim}
\label{Main rigidity result: Claim 1}
Suppose $\cT$ pins $\AgeK$, $\cF \subseteq \MBot$ with $\cF \in V$ and $\cF$ has splitting property over $\MBot$ and $\cT$. Then $V[G] \models \name{h}^G\rest[{\cF}] \neq \id$
\end{claim}
\begin{proof}
Suppose to get a contradiction that $\name{h}^{G}\rest[{\cF}] = \id$. Let $s, t \in \cT^{\MBot}$ be distinct with $\name{h}^{G}(s) = t$. 

Let $\cB \in G$ such that $\cA \subseteq \cB$ and $\cB \forces[\Fn<\MBot>[\AgeK](\lambda)] \name{h}(s) = t$. Let $\cA^- = \cA\rest[\LangBot]$. 

Let $E$ be the collection of $\cE \in \AgeK[\MBot]$ such that 
\begin{itemize}

\item $\cB \subseteq \cE$,

\item $s, t \in \cE$, 

\item $(\exists \zz_{\cE} \in \cE \cap \cF)\, 
\qftp[\cE](s\zz_{\cE}) \neq \qftp[\cE](t\zz_{\cE})\}$. 
\end{itemize}
Then $E$ is dense below $\cB$ because $\cF$ has the splitting property with respect to $\MBot$ and $\cT$ and $s, t \in \cT^{\MBot}$. Therefore there must be some  $\cD \in G \cap E$. But then $\qftp[\cG](s\zz_{\cD}) = \qftp[\cG](\name{h}^G(s)\name{h}^G(\zz_{\cD})) = \qftp[\cE](t\zz_{\cD})$, contradicting the fact that $\cD \in E$. 
\end{proof}

By the claim, for each $i \in \kappa^+$, there are $\cB_i \in \AgeK$, $s_i \in \MBot$ and $t_i \in \MBot$ such that $\cA \subseteq \cB_i$, $s_i \in \cF_i$ and 
\[
\cB_i \forces[\Fn<\MBot>[\AgeK](\lambda)] \name{h}(s_i) = t_i\text{ and }t_i \neq s_i.
\]
Further as the $\cF_i$ are disjoint outside $\tdcl[\MBot](\emptyset)$ and $\tdcl[\MBot](\emptyset)$ is fixed by any automorphism of $\MBot$ we must have $s_i \neq s_j$ for $i < j \in \kappa^+$. 

But then by \cref{Delta-Systems Lemma} we can find an $I \subseteq \kappa^+$ with $|I| = \kappa^+$ such that for all $i < j \in I$, $\cA, \cB_i, \cB_j$ satisfy \cref{Amalgamation of distinct types property} (b). Further, as there are at most $\kappa$-many extensions of $\cA$ up to isomorphism  we can assume $I$ was chosen so that \cref{Amalgamation of distinct types property} (c) holds as well. Let $\alpha_{i, j}$ be an isomorphism from $\cB_i$ to $\cB_j$ guaranteed by \cref{Amalgamation of distinct types property} (c). As each $\cB_i$ must have cardinality $\leq \kappa$, we can further assume without loss of generality that $\alpha_{i, j}(s_i) = s_j$ and $\alpha_{i, j}(t_i) = t_j$ for all $i, j \in I$.

For $i, j \in I$ let $\DBot_{i, j}$ be the substructure of $\MBot$ generated by $\UnderSet{\cB_i} \cup \UnderSet{\cB_j}$. Then \cref{Amalgamation of distinct types property} (d) and (e) hold. Therefore there must be a $\cD  \in \AgeK[\MBot]$ such that $\cD\rest[\LangBot] = \DBot$, $\cB_i \subseteq \cD$, $\cB_j \subseteq \cD$ and $\qftp[\cD](s_is_j) \neq \qftp[\cD](t_it_j)$. 

But then 
\[
\cD \forces[\Fn<\MBot>[\AgeK](\lambda)] \name{h}\text{ is not an isomorphism }
\]
getting us our contradiction. 
\end{proof}

We now give some examples of how to apply \cref{Main rigidity result}

\begin{example}
\label{Example orientation of 2-dimensional subspace} 
Suppose $\Field$ is a finite field. Let $\LangBot_{\Field}$ and $\AgeKBot_{\Field}$ be the language and age in \cref{Example: Colored vector spaces} associated to finite vector spaces over $\Field$. 

For $\kappa$ a cardinal let $\MBot_{\Field, \kappa}$ be the $\LangBot_{\Field}$-structure with age $\AgeKBot_{\Field}$ which is a $\kappa$-dimensional $\Field$-vector space. Note $\MBot_{\Field, \kappa} \models \ThFraisse_\w(\AgeKBot_{\Field})$ whenever $\kappa$ is infinite. 

Let $\Lang = \LangBot_{\Field} \cup \{R\}$ where $R$ is a binary relation. Let $\AgeK$ be the age in $\Lang$ such that for all $\cA \in \AgeK$
\begin{itemize}
\item $\cA\rest[\LangBot_{\Field}] \in \AgeKBot_{\Field}$, 

\item For every embedding $i\:\MBot_{\Field, 2} \to \cA\rest[\LangBot_{\Field}]$ there is a unique pair of independent elements $a, b \in \MBot_{\Field, 2}$ such that $\cA \models R(i(a), i(b))$. 
\end{itemize}

We can think of $\AgeK$ as the age of vector spaces over $\Field$ where we have, for ever $2$-dimensional subspace, a distinguished orientation. 

Note it is straightforward to check that both $\AgeKBot_{\Field}$ and $\AgeK$ are strong \Fraisse\ classes and that $\AgeK$ has the extended strong amalgamation property over $\LangBot_{\Field}$. 

To show that $\AgeK$ has the amalgamation of distinct $2$-types property we now let $\cA, \cB, \cC, \DBot, \alpha, x_0, x_1, \cS$ be as in \cref{Amalgamation of distinct types property} where $\cS$ is the unique sort and $\LangBot = \LangBot_{\Field}$. Let $W_0$ be the vector space generated by $x_0\alpha(x_0)$ and $W_1$ be the vector space generated by $x_1\alpha(x_1)$. Note that $W_0 \cap \cB$ consists of the $1$-dimensional vector space containing $x_0$ and similarly $W_1 \cap \cB$ consists of the $1$-dimensional vector space containing $x_1$. Therefore $R$ does not holds of any tuple in $W_0 \cap \cB$ nor of any tuple in $W_1 \cap \cB$. Therefore there must be some $\cD \in \AgeK[\cD^-]$ where $\cD \models R(x_0\alpha(x_0))$ and $\cD \models \neg R(x_1\alpha(x_1))$. Such a $\cD$ witnesses the amalgamation of distinct $2$-types property.   

Now let $(\cF_i)_{i \in \kappa^+}$ be a collection of infinite dimensional subspaces of $\MBot_{F, \kappa^+}$ such that for $i < j \in \kappa^+$, $\cF_i \cap \cF_j = \{v_0^{\MBot_{F, \kappa^+}}\}$. It is immediate that each $\cF_i$ has the splitting property for amalgamation property (with respect to $\cS$). 

Therefore, by \cref{Main rigidity result}, if $G$ is a generic for $\Fn<\MBot_{\Field, \kappa^+}>[\AgeK](\w)$ over $V$, $\bigcup G$ is rigid in $V[G]$. In other words if we choose the orientations of $2$-dimensional subspaces generically the result is a rigid structure.  
\end{example}

Note \cref{Example orientation of 2-dimensional subspace} was only done for $2$-dimensional subspaces. This is because the conditions which go into \cref{Main rigidity result} only talk about pairs of elements. The following example gives us a way, sometimes, to get around this limitation. In particular this will allow us to show if $R$ is a $n$-ary relation (with $n \geq 5$),  $\AgeK_R$ is the age consisting of all finite structures in the language $\{R\}$, and $\NBot$ is the $\kappa$-sized structure in the empty language than any generic for $\Fn<\NBot>[\AgeK_R](\w)$ gives rise to a rigid structure. This generalizes Theorem 4 of \cite{Cohen} to the higher arity case. 

\begin{example} 
\label{Example of 5-ary directed hypergraphs}
Let $n \geq 5$ and $\LangBot_{s, n}$ and $\AgeKBot_{s, n}$ be as in \cref{Exaple of names for sequences}. Let $\Lang_{s,n} = \LangBot_{s, n} \cup \{R\}$ where $R$ is a relation of arity $U^n$. Let $\AgeK<s,n>$ be the age in $\Lang_{s, n}$ consisting of all elements whose restriction to $\LangBot_{s, n}$ is in $\AgeKBot_{s, n}$. Elements of $\AgeK_{s, n}$ can be thought of as arbitrary structures in the language $\{R\}$ along with names for each sequence of length $n$. 

It is immediate that $\AgeK<s, n>$ is a strong \Fraisse\ class as $\AgeKBot_{s, n}$ is a strong  \Fraisse\ class. It is also immediate that $\AgeK<s,n>$ has the extended strong amalgamation property over $\Lang_{s,n}$. Also, note that the sort $W$ pins $\AgeKBot_{s, n}$. 

If $\MBot_{s, n}$ is a $\LangBot_{s, n}$-structure such that $\MBot_{s, n} \models \ThFraisse_\w(\AgeKBot_{s, n})$ then $\MBot$ is completely determined, up to isomorphism, by its cardinality. For $\kappa$ an infinite cardinal let $\MBot_{\kappa} \models \ThFraisse_\w(\AgeKBot_{s, n})$ be the unique model (up to isomorphism) where $|\MBot_{\kappa}| = \kappa$. 

To show that $\AgeK<s, n>$ has the amalgamation of distinct $2$-types property we now let $\cA, \cB, \cC, \DBot, \alpha, x_0, x_1, \cS$ be as in \cref{Amalgamation of distinct types property} where $\cS = W$ and $\LangBot = \LangBot_{s, n}$. For $i \in [n]$ let $y_i = \pi_i(x_0)$ and $z_i = \pi_i(x_1)$. There must be some $\ell \in [n]$ such that $y_\ell \neq z_\ell$. As $n \geq 5$ we can assume without loss of generality that $\ell > 3$. There also must be some $k_0, k_1 \in [n]$ such that $\alpha(y_{k_0}), \alpha(z_{k_1}) \not \in \cB$ as otherwise $\alpha(x_{0}), \alpha(y_{1}) \in \cB$. This then implies that $y_{k_0}, z_{k_1} \not \in \cC$.  
 
Define 
\[
a = f(\alpha(y_{k_0}), \alpha(y_{k_1}), y_{k_0}, y_{k_1}, \dots, y_\ell, \dots, y_{n-1})
\]
and define 
\[
b = f(\alpha(z_{k_0}), \alpha(z_{k_1}), z_{k_0}, z_{k_1}, \dots, z_\ell, \dots, z_{n-1}).
\]
Note that $a \neq b$ as $y_\ell \neq z_\ell$. Further $a \not \in \{x_0, x_1\}$ as $\alpha(y_{k_0}) \not \in \cB$ and $a \not \in\{\alpha(x_0),\alpha(x_1)\}$ as $y_{k_0} \not \in \cC$. We have similar conclusions about $b$. 

Therefore there is some $\cD \in \AgeK<s, n>[\DBot]$ with $\cB \subseteq \cD$, $\cC \subseteq \cD$ and 
\[
\cD \models R(\alpha(y_{k_0}), \alpha(y_{k_1}), y_{k_0}, y_{k_1}, \dots, y_\ell, \dots, y_{n-1}) \And\neg R(\alpha(z_{k_0}), \alpha(z_{k_1}), z_{k_0}, z_{k_1}, \dots, z_\ell, \dots, z_{n-1}).
\]
Therefore $\AgeK<s, n>$ has the amalgamation of distinct $2$-types property. 

Note a more detailed analysis of which elements of $a$ and $b$ were in $\cA$ would have shown that $\AgeK<s,4>$ and $\AgeK<s, 3>$ had amalgamation of distinct $2$-types property. However such an analysis would not have added much to the example. In contrast, $\AgeK<s, 2>$ does not have amalgamation of distinct $2$-types property. To see this consider $x_0 = f(y_0, y_1)$ and $x_1 = f(y_0, z_1)$ with $y_1, z_1 \in \cA$. Then $\qftp[\cD](y_0, y_1, \alpha(y_0), \alpha(y_1)) = \qftp[\cD](y_0, y_1, \alpha(y_0), y_1)$  which is completely determined by $\qftp[\cD](y_0, \alpha(y_0))$, $\cB$ and $\cC$. But we similarly have $\qftp[\cD](y_0, z_1, \alpha(y_0), \alpha(z_1))= \qftp[\cD](y_0, z_1, \alpha(y_0), z_1)$ which is also completely determined by $\qftp[\cD](y_0, \alpha(y_0))$, $\cB$ and $\cC$. 

Let $P$ be a partition of $U^{\MBot_{s, \kappa^+}}$ into countable infinite sets. For $p \in P$ let $\cF_p = \{f(x_0, \dots, x_{n-1}) \st x_0 \in p\And x_1, \dots, x_{n-1} \in U^{\MBot_{s, \kappa^+}}\}$. Now let $\cA \in \AgeK<s, n>[\MBot_{s\kappa^+}]$ and let $a_0, a_1 \in W^{\cA}$ be distinct. As $\cA$ is finite, given any $p \in P$ we can find an $b \in \cF_p$ such that $\pi_0(b) \not \in \cA$. We can therefore find a $\cB \in \AgeK<s, n>[\MBot_{s, \kappa^+}]$ with $\cA \subseteq \cB$ such that 
\[
\cB \models R(\pi_0(b), \pi_1(a_0), \dots, \pi_{n-1}(a_{0})) \And \neg R(\pi_0(b), \pi_1(a_1), \dots, \pi_{n-1}(a_{1})). 
\]
Therefore $\cF_p$ has the splitting property over $\AgeK<s, n>$ and $W$. Further, for $p_0, p_1 \in P$ distinct, we have $\cF_{p_0} \cap \cF_{p_1} = \emptyset$. Therefore, as $|P| = \kappa^+$, we can apply \cref{Main rigidity result} to get  whenever $G$ is a generic for $\Fn<\MBot_{s, \kappa^+}>[\AgeK<s, n>](\w)$ over $V$, $\bigcup G$ is rigid in $V[G]$.

Finally, suppose $\LangBot_* = \{U\}$ where $U$ is a sort and $\Lang_* = \LangBot_* \cup \{R\}$ where $R$ is a relation of type $U^n$.  If $\MBot$ is an $\LangBot_*$ structure then there is a unique expansion of $\MBot$ to a $\LangBot_{s, n}$-structure whose age is in $\AgeKBot_{s, n}$. Similarly if $\cM$ is a $\Lang_*$-structure then there is a unique expansion of $\cM$ to a $\Lang_{n, s}$-structure whose age is in $\AgeK<n, s>$. Therefore if $\AgeK_R$ is the age consisting of all $\Lang_*$-structures and $\NBot_{\kappa^*}$ is the unique $\LangBot_*$-structure of size $\kappa^+$, whenever $G$ is a generic for $\Fn<\NBot_{\kappa^+}>[\AgeK_R](\w)$ over $V$, $\bigcup G$ is rigid in $V[G]$. 
\end{example}

In \cref{Example of 5-ary directed hypergraphs}, we showed that forcing with arbitrary structures in a relation of arity $\geq 5$ results in a rigid structure in the generic extension. This required us to have access to function symbols to name the tuples and so, even though the structures we are interested in, i.e. $\Lang_*$-structures, are in a relational language, we needed to be able to deal with function symbols to show the rigidity (at least if we wanted to use these techniques). 

Note though that we dealt with arbitrary $\Lang_*$-structures and the techniques we applied do not immediately generalize to deal with symmetric $n$-hypergraphs. There are two natural approaches one might try to generalized these techniques. First one could have  names for all tuples of length $n$ and then simply require that whenever the relation holds it must hold of all permutations of its arguments. This however would cause a problem as if $a, b$ were names of permutations of the same tuple there would be no way to distinguish them in the structure and amalgamation of distinct $2$-types property would fail.

The second natural approach one might try is to provide elements which ``name'' a set of size $n$, instead of sequences. However in this case there is no obvious way to recover the elements of the set from the name using just function symbols. This is because if $y$ is a name for $\{x_0, \dots, x_{n-1}\}$ and $x_0 = t(y)$ for some term, then any automorphism of the structure which fixes $y$ also fixes $x_0$. This therefore adds structure to the set of elements $\{x_0, \dots, x_{n-1}\}$. 

We therefore have the following conjecture.  

\begin{conjecture}
Suppose $n \geq 2$. Let $\LangBot$ be the empty language and let $\Lang_n = \{R\}$ where $R$ is a relation of arity $n$. Let $\AgeK<sym,n>$ consist of all those finite $\Lang_n$-structures where $R$ is symmetric, i.e. which satisfy 
\[
\bigwedge_{\tau \in \Sym{[n]}} (\forall x_0, \dots, x_{n-1})\, R(x_0, \dots, x_{n-1}) \leftrightarrow R(x_{\tau(0)}, \dots, x_{\tau(n-1)}).
\]

Then for any uncountable cardinal $\kappa$, if $\MBot_{\kappa}$ is the unique $\LangBot$-structure of size $\kappa$ then whenever $G$ is a generic for $\Fn<\MBot_{\kappa}>[\AgeK<{sym, n}>](\w)$ over $V$, $\bigcup G$ is rigid in $V[G]$. 
\end{conjecture}

\section{Non-Rigidity}  


While \cref{Main rigidity result} gives us several examples of ages such that any generic over an uncountable structure is rigid, there are also many examples where the generics have large automorphism groups. The following proposition gives us a collection of such ages. 

\begin{definition}
Suppose $\Lang$ is a relational language and let $\Phi$ be the collection of atomic formulas. If $\cM$ is an  $\Lang$-structure and  $a, b \in \cM$ say $\cM \models a \equiv b$ if 
\[
\cM \models (\forall \cc)\bigwedge_{\psi(x,\yy) \in \Phi} \psi(a, \cc) \leftrightarrow \psi(b, \cc). 
\]
\end{definition}

Note the following is immediate. 
\begin{lemma}
Suppose $\Lang$ is a relational language and $\cM$ is an $\Lang$-structure. Suppose $A \subseteq\cM$ is an $\equiv^{\cM}$-equivalence class. Then for any bijection $f\:A \to A$ the map $f \cup \id_{\cM \setminus A}$ is an automorphism of $\cM$. 
\end{lemma}

\begin{corollary}
\label{Conditions when S_k is in the automorphism group}
Suppose $\Lang$ is a relational language and  $\cM$ is an $\Lang$-structure. If $\cM$ has an $\equiv^{\cM}$-equivalence class of size $\kappa$ then there is a subgroup of $\Aut(\cM)$ isomorphic to $\Perm(\kappa)$.  
\end{corollary}

\begin{proposition}
Suppose 
\begin{itemize}
\item $\Lang$ is a countable relational language,

\item $\AgeK$ is a strong \Fraisse\ class which is $\w$-bounded and which has at most $\w$-many types up to isomorphism. 

\item $\cM$ is the \Fraisse\ limit of $\AgeK$ and $\equiv^{\cM}$ is not equality. 

\item for $\kappa$ an infinite cardinal, $\cN_\kappa$ is the unique structure on $\kappa$ in the empty language. 
\end{itemize}
Then if $G$ is any generic over $V$ for $\Fn<\cN_\kappa>[\AgeK](\w)$ we have $\Perm(\kappa)$ is a subgroup of $\Aut(\bigcup G)$. 
\end{proposition}
\begin{proof}
By \cref{Fraisse limits exist and trivial group theoretic dcl} and the fact that $\AgeK$ is a strong \Fraisse\ class, $\cM$ has trivial group theoretic definable closure. Now suppose to get a contradiction $\equiv^{\cM}$ had a finite equivalence class $\{a_i\}_{i \in [n]}$ where $n > 1$. Because being in the same $\equiv^{\cM}$-equivalence class is preserved under automorphisms, if $g \in \Aut(\cM)$ and $g(a_0, \dots, a_{n-2}) = (a_0, \dots, a_{n-2})$ we must also have $g(a_{n-1}) = a_{n-1}$. However this would contradict the fact that $\cM$ has trivial group theoretic definable closure. Therefore every $\equiv^{\cM}$-equivalence class must either have size $1$ or be infinite. Furthermore, as $\equiv^{\cM}$ is not equality it must have an infinite equivalence class. 

But by \cref{Forcing satisfies Fraisse theory} we know that $\bigcup \cG$ satisfies the \Fraisse\ theory of $\AgeK$ and so is potentially isomorphic to $\cM$. Therefore there must be some $\cA \in \AgeK$ and $a \in \cA$ such that 
\[
\cA \forces[\Fn<\cN_\kappa>[\AgeK](\w)] a\text{ is in an infinite }\equiv\text{-equivalence class}. 
\]

But then for all $\alpha \in \kappa$, the collection of $\cB$ such that $\cA \subseteq \cB$ and 
\[
\cB \forces[\Fn<\cN_\kappa>[\AgeK](\w)] (\exists \beta > \alpha)\, a \equiv \beta.
\]
is dense. Therefore in $\bigcup G$ we have that $a$ is in an $\equiv$-equivalence class of size $\kappa$. 

But, by \cref{Conditions when S_k is in the automorphism group}, we then have that $\Perm(\kappa)$ is a subgroup of $\Aut(\bigcup G)$ as desired.  
\end{proof}

%
%
%
%
%

\bibliographystyle{amsnomr}
\bibliography{bibliography}

\end{document}